\documentclass[12pt]{article}
\usepackage{amssymb}
\usepackage{amsthm}
\usepackage{amsmath}
\usepackage{graphicx}
\usepackage{enumerate}

\newcommand{\dotcup}{\ensuremath{\mathaccent\cdot\cup}}
\newtheorem{theorem}{Theorem}
\newtheorem{definition}[theorem]{Definition}
\newtheorem{lemma}[theorem]{Lemma}

\newtheorem{example}[theorem]{Example}

\title{Unsharp residuation in effect algebras}
\author{Ivan~Chajda and Helmut~L\"anger}
\date{}
\begin{document}
\footnotetext[1]{Support of the research by \"OAD, project CZ~02/2019, and support of the research of the first author by IGA, project P\v rF~2019~015, is gratefully acknowledged.}
\maketitle
\begin{center}
Dedicated to the memory of Ivo~G.~Rosenberg
\end{center}
\begin{abstract}
Effect algebras and pseudoeffect algebras were introduced by Foulis, Bennett, Dvure\v censkij and Vetterlein as so-called quantum structures which serve as an algebraic axiomatization of the logic of quantum mechanics. A natural question concerns their connections to substructural logics which are described by means of residuated lattices or posets. In a previous paper it was shown that an effect algebra can be organized into a so-called conditionally residuated structure where the adjointness condition holds only for those elements for which the operations $\odot$ and $\rightarrow$ are defined. Because this is a very strong restriction, we try to find another kind of residuation where the terms occurring in the adjointness condition are everywhere defined though the binary operation of a given effect algebra is only partial. Moreover, we work with effect algebras which need not be lattice-ordered and hence the lattice operations join and meet are replaced by means of upper and lower cones which, however, are not elements but subsets. Hence, the resulting concept, the so-called unsharp residuated poset, is equipped with LU-terms which substitute operations, but are everywhere defined. Although this concept seems rather complicated at the first glance, we prove that such an unsharp residuated poset can be conversely organized into an effect algebra or a pseudoeffect algebra depending on commutativity of the multiplication.
\end{abstract}
 
{\bf AMS Subject Classification:} 03G25, 03B47, 06A11, 06F35

{\bf Keywords:} Unsharp residuation, residuated poset, effect algebra, monotonous effect algebra, pseudoeffect algebra, good pseudoeffect algebra

The logics of quantum mechanics is based on so-called effects which form an effect algebra as shown by Foulis and Bennett (\cite{FB}). This algebra is a partial algebra with one partial binary associative and commutative operation. Moreover, every element of this partial algebra has a unique so-called supplement. In effect algebras there can be introduced a partial order relation by means of the partial binary operation in some natural way. If this relation is a lattice order then the corresponding effect algebra is called a lattice effect algebra. The question arises if there is a connection between effect algebras and substructural logics where conjunction and implication are adjoint to each other, i.e.\ logics which form a residuated lattice, see e.g.\ \cite{GJKO}. The authors showed in \cite{CL2} that in the case of lattice effect algebras this is possible if one-sided (so-called left residuation) is used. A similar approach was settled by the authors also for other quantum structures, e.g.\ for orthomodular lattices, see \cite{CL17a} and \cite{CL17b}. Unfortunately, in the logics of quantum mechanics, lattice effect algebras play only a limited role. Much more useful are those effect algebras that are not lattice-ordered. A certain solution was already published in \cite{CH}, but within this paper residuation holds only in the case when the corresponding terms occurring in adjointness are defined. Since this is a rather strict restriction, we are going to introduce our concepts in a more acceptable form. Hence the natural question arises if some kind of residuation can also be introduced in effect algebras that are not lattices. The problem is that in non-lattice ordered effect algebras disjunction and conjunction cannot be defined by join and meet, respectively, since the latter need not exist. Hence also implication cannot be introduced in a sharp sense. Therefore it is of interest if one can apply some unsharp approach where supremum and infimum are replaced by the upper and lower cone, respectively. In the present paper we show that this really is possible.

Let $(P,\leq)$ be a poset, $a,b\in P$ and $A,B\subseteq P$. We put
\begin{align*}
L(A) & :=\{x\in P\mid x\leq y\text{ for all }y\in A\}, \\
U(A) & :=\{x\in P\mid y\leq x\text{ for all }y\in A\}.
\end{align*}
Instead of $L(\{a\})$, $L(\{a,b\})$, $L(A\cup\{a\})$, $L(A\cup B)$ and $L(U(A))$ we simply write $L(a)$, $L(a,b)$, $L(A,a)$, $L(A,B)$ and $LU(A)$, respectively. Analogously, we proceed in similar cases.

For the reader's convenience, we repeat the definition of effect algebra \cite{FB}.

\begin{definition}
An {\em effect algebra} is a partial algebra $\mathbf E=(E,+,{}',0,1)$ of type $(2,1,0,0)$ where $(E,{}',0,1)$ is an algebra and $+$ is a partial operation satisfying the following conditions for all $x,y,z\in E$:
\begin{enumerate}[{\rm(E1)}]
\item $x+y$ is defined if and only if so is $y+x$ and in this case $x+y=y+x$,
\item $(x+y)+z$ is defined if and only if so is $x+(y+z)$ and in this case $(x+y)+z=x+(y+z)$,
\item $x'$ is the unique $u\in E$ with $x+u=1$,
\item if $1+x$ is defined then $x=0$.
\end{enumerate}
On $E$ a binary relation $\leq$ can be defined by
\[
x\leq y\text{ if there exists some }z\in E\text{ with }x+z=y
\]
{\rm(}$x,y\in E${\rm)}. Then $(E,\leq,0,1)$ becomes a bounded poset and $\leq$ is called the {\em induced order} of $\mathbf E$. If $(E,\leq)$ is a lattice then $\mathbf E$ is called {\em lattice-ordered}.
\end{definition}

There follow two examples of effect algebras which are not lattice ordered.

\begin{example}\label{ex1}
Let $E$ denote the set $\{0,a,b,c,d,e,f,g,1\}$ and define $+$ and ${}'$ as follows:
\[
\begin{array}{c|ccccccccc}
+ & 0 & a & b & c & d & e & f & g & 1 \\
\hline
0 & 0 & a & b & c & d & e & f & g & 1 \\
a & a & - & e & f & - & - & - & 1 & - \\
b & b & e & d & g & f & - & 1 & - & - \\
c & c & f & g & - & - & 1 & - & - & - \\
d & d & - & f & - & 1 & - & - & - & - \\
e & e & - & - & 1 & - & - & - & - & - \\
f & f & - & 1 & - & - & - & - & - & - \\
g & g & 1 & - & - & - & - & - & - & - \\
1 & 1 & - & - & - & - & - & - & - & -
\end{array}
\quad
\begin{array}{c|c}
x & x' \\
\hline
0 & 1 \\
a & g \\
b & f \\
c & e \\
d & d \\
e & c \\
f & b \\
g & a \\
1 & 0
\end{array}
\]
Then $\mathbf E:=(E,+,{}',0,1)$ is an effect algebra that is not lattice ordered. Its induced poset is depicted in Figure~1:

\vspace*{-2mm}

\[
\setlength{\unitlength}{7mm}
\begin{picture}(6,9)
\put(3,2){\circle*{.3}}
\put(1,4){\circle*{.3}}
\put(3,4){\circle*{.3}}
\put(5,4){\circle*{.3}}
\put(3,5){\circle*{.3}}
\put(1,6){\circle*{.3}}
\put(3,6){\circle*{.3}}
\put(5,6){\circle*{.3}}
\put(3,8){\circle*{.3}}
\put(3,2){\line(-1,1)2}
\put(3,2){\line(0,1)6}
\put(3,2){\line(1,1)2}
\put(1,6){\line(0,-1)2}
\put(1,6){\line(1,-1)2}
\put(1,6){\line(1,1)2}
\put(5,6){\line(0,-1)2}
\put(5,6){\line(-1,-1)2}
\put(5,6){\line(-1,1)2}
\put(3,6){\line(-1,-1)2}
\put(3,6){\line(1,-1)2}
\put(2.875,1.25){$0$}
\put(.35,3.85){$a$}
\put(3.4,3.85){$b$}
\put(5.4,3.85){$c$}
\put(3.4,4.85){$d$}
\put(.35,5.85){$e$}
\put(3.4,5.85){$f$}
\put(5.4,5.85){$g$}
\put(2.85,8.4){$1$}
\put(2.2,.3){{\rm Fig.~1}}
\end{picture}
\]

\vspace*{-3mm}

\end{example}              

\begin{example}\label{ex2}
Let $E$ denote the set of all subsets of $\{1,\ldots,6\}$ of even cardinality and define
\begin{align*}
A+B & :=A\cup B\text{ if and only if }A\cap B=\emptyset, \\
 A' & :=\{1,\ldots,6\}\setminus A
\end{align*}
{\rm(}$A,B\in E${\rm)}. Then $\mathbf E=(E,+,{}',\emptyset,\{1,\ldots,6\})$ is an effect algebra that is not a lattice-ordered.
\end{example}

We mention some concepts from posets which we will need in the sequel.

Let $(P,\leq)$ be a poset, $a\in P$ and $A,B\subseteq P$. We define $A\leq B$ if and only if $x\leq y$ for all $x\in A$ and all $y\in B$. Instead of $A\leq\{a\}$ we simply write $A\leq a$. Analogously, we proceed in similar cases.

In the sequel we will use the properties of effect algebras listed in the following lemma.

\begin{lemma}\label{lem1}
{\rm(}see {\rm\cite{DV}, \cite{FB})} If $\mathbf E=(E,+,{}',0,1)$ is an effect algebra, $\leq$ its induced order and $a,b,c\in E$ then
\begin{enumerate}[{\rm(i)}]
\item $a''=a$,
\item $a\leq b$ implies $b'\leq a'$,
\item $a+b$ is defined if and only if $a\leq b'$,
\item if $a\leq b$ and $b+c$ is defined then $a+c$ is defined and $a+c\leq b+c$,
\item if $a\leq b$ then $a+(a+b')'=b$,
\item $a+0=0+a=a$,
\item $0'=1$ and $1'=0$.
\end{enumerate}
\end{lemma}

\begin{definition}
An {\em effect algebra} $(E,+,{}',0,1)$ is called {\em monotonous} if it satisfies the following condition for all $x\in E$ and all non-empty subsets $A,B$ of $E$:
\[
A\cup B\leq x'\text{ and }L(A)\leq U(B)\text{ imply }L(x+A)\leq U(x+B).
\]
Here and in the following $x+A$ denotes the set $\{x+y\mid y\in A\}$.
\end{definition}

Since $A\cup B\leq x'$, $x+A$ and $x+B$ are defined.

\begin{lemma}
Let $(E,+,{}',0,1)$ be a monotonous effect algebra, $a\in E$ and $A,B$ non-empty subsets of $E$ and assume $a'\leq A\cup B$ and $L(A)\leq U(B)$. Then $L(a\odot A)\leq U(a\odot B)$.
\end{lemma}

\begin{proof}
We have $A'\cup B'\leq a$ and
\[
L(B')=(U(B))'\leq(L(A))'=U(A')
\]
and hence
\[
L(a\odot A)=L((a'+A'))=(U(a'+A'))'\leq(L(a'+B'))'=U((a'+B'))=U(a\odot B).
\]
\end{proof}

\begin{lemma}
Let $\mathbf E=(E,+,{}',0,1)$ be an effect algebra satisfying the following conditions for all $x,y\in E$ and all non-empty subsets $A$ of $E$:
\begin{itemize}
\item If $A\leq y'$ and $x\in L(y+A)$ then there exist $z\in L(y)$ and $u\in L(A)$ with $z+u=x$,
\item if $A\leq y'$ and $x\in U(y+A)$ then there exist $z\in U(y)$ and $u\in U(A)$ with $z+u=x$.
\end{itemize}
Then $\mathbf E$ is monotonous.
\end{lemma}

\begin{proof}
Let $a\in E$ and $A,B$ be non-empty subsets of $E$ and assume $A\cup B\leq a'$, $L(A)\leq U(B)$, $c\in L(a+A)$ and $d\in U(a+B)$. Then there exist $e\in L(a)$, $f\in L(A)$, $g\in U(a)$ and $h\in U(B)$ with $e+f=c$ and $g+h=d$. Hence $e\leq a\leq g$ and $f\leq h$ and therefore $c=e+f\leq g+h=d$.
\end{proof}

One can easily check that the effect algebra from Example~\ref{ex1} is monotonous. The same is true for the effect algebra $\mathbf E$ from Example~\ref{ex2}. This can be seen as follows:

If $I,J\neq\emptyset$ and $A,A_i,B_j\in E$ for all $i\in I$ and all $j\in J$,
\[
\bigcup_{i\in I}A_i\cup\bigcup_{j\in J}B_j\subseteq A'
\]
and
\[
L(\{A_i\mid i\in I\})\leq U(\{B_j\mid j\in J\})
\]
then
\[
\bigcap_{i\in I}A_i\subseteq\bigcup_{j\in J}B_j
\]
and therefore
\[
\bigcap_{i\in I}(A\dotcup A_i)=A\dotcup\bigcap_{i\in I}A_i\subseteq A\dotcup\bigcup_{j\in J}B_j=\bigcup_{j\in J}(A\dotcup B_j)
\]
which shows
\[
L(\{A\dotcup A_i\mid i\in I\})\leq U(\{A\dotcup B_j\mid j\in J\}).
\]

A {\em partial monoid} is a partial algebra $\mathbf A=(A,\odot,1)$ of type $(2,0)$ where $\odot$ is a partial operation satisfying the following conditions for all $x,y,z\in A$:
\begin{enumerate}[{\rm(i)}]
\item $(x\odot y)\odot z$ is defined if and only if so is $x\odot(y\odot z)$ and in this case $(x\odot y)\odot z=x\odot(y\odot z)$,
\item $x\odot1\approx1\odot x\approx x$.
\end{enumerate}
The {\em partial monoid} $\mathbf A$ is called {\em commutative} if it satisfies the following condition for all $x,y\in A$:
\begin{enumerate}
\item[(iii)] $x\odot y$ is defined if and only if so is $y\odot x$ and in this case $x\odot y=y\odot x$,
\end{enumerate}

Now we are ready to define our main concept.

\begin{definition}
A {\em commutative unsharp residuated poset} is an ordered seventuple $\mathbf C=(C,\leq,\odot,\rightarrow,{}',0,1)$ where $\rightarrow:C^2\rightarrow2^C$ and the following hold for all $x,y,z\in C$:
\begin{enumerate}[{\rm(C1)}]
\item $(C,\leq,{}',0,1)$ is a bounded poset with an antitone involution,
\item $(C,\odot,1)$ is a partial commutative monoid where $x\odot y$ is defined if and only if $x'\leq y$. Moreover, $z'\leq x\leq y$ implies $x\odot z\leq y\odot z$,
\item $L(U(x,y')\odot y)\leq UL(y,z)$ if and only if $LU(x,y')\leq U(y\rightarrow z)$,
\item $x\rightarrow0\approx\{x'\}$,
\end{enumerate}
Condition {\rm(C3)} is called {\em unsharp adjointness}. The commutative unsharp residuated poset $\mathbf C$ is called {\em divisible} if $x\leq y$ implies $y\odot(y\rightarrow x)=L(x)$ {\rm(}$x,y\in C${\rm)}.
\end{definition}

Note that because of $y'\leq U(x,y')$ the expression $U(x,y')\odot y$ is well-defined.

Let us note that in the definition of unsharp adjointness we have an additional element $y$ in the term $UL(y,z)$ on the right-hand side of the first inequality and an additional element $y'$ in the term $U(x,y')$ on the left-hand side of both inequalities. In our paper \cite{CL1} (where, however, adjointness is not unsharp) the corresponding {\em residuation} is called {\em relative}. This means that it is ``relative to $y$''.

Using our concept of commutative unsharp residuated poset, we can prove the following conversion of a monotonous effect algebra into this kind of residuated poset.

\begin{theorem}\label{th1}
Let $\mathbf E=(E,+,{}',0,1)$ be a monotonous effect algebra with induced order $\leq$ and put
\begin{align*}
      x\odot y & :=(x'+y')'\text{ if and only if }x'\leq y, \\
x\rightarrow y & :=x'+L(x,y)
\end{align*}
{\rm(}$x,y\in E${\rm)}. Then $\mathbb C(\mathbf E):=(E,\leq,\odot,\rightarrow,{}',0,1)$ is a divisible commutative unsharp residuated poset.
\end{theorem}

\begin{proof}
Let $a,b,c\in E$. Obviously, (C1) and the first part of (C2) hold.
\begin{enumerate}
\item[(C2)] If $c'\leq a\leq b$ then
\[
a\odot c=(a'+c')'\leq(b'+c')'=b\odot c.
\]
\item[(C3)] If $L(U(a,b')\odot b)\leq UL(b,c)$ then because of $b'\leq U(a,b')$, $U(a,b')\odot b\leq b$ and $L(b,c)\leq b$ we have
\[
LU(a,b')=L(b'+(U(a,b')\odot b))\leq U(b'+L(b,c))=U(b\rightarrow c).
\]
If, conversely, $LU(a,b')\leq U(b\rightarrow c)$ then because of $b'\leq U(a,b')$, $b'\leq b\rightarrow c$ and $L(b,c)\leq b$ we have
\[
L(U(a,b')\odot b)\leq U((b\rightarrow c)\odot b)=U((b'+L(b,c))\odot b)=U(((b'+L(b,c))'+b')')=UL(b,c).
\]
\item[(C4)] We have $x\rightarrow0\approx x'+L(x,0)\approx\{x'\}$.
\end{enumerate}
If, finally, $a\leq b$ then $L(a)\leq b$ and we have
\[
b\odot(b\rightarrow a)=b\odot(b'+L(b,a))=b\odot(b'+L(a))=(b'+(b'+L(a))')'=L(a)
\]
proving divisibility.
\end{proof}

\begin{example}
Since the effect algebras from Example~\ref{ex1} and \ref{ex2} are monotonous, by Theorem~\ref{th1}, they can be converted into divisible commutative unsharp residuated posets. For the effect algebra from Example~\ref{ex1}, $\odot$ and $\rightarrow$ are given by the following tables:
\[
\begin{array}{c|ccccccccc}
\odot & 0 & a & b & c & d & e & f & g & 1 \\
\hline
  0   & - & - & - & - & - & - & - & - & 0 \\
	a   & - & - & - & - & - & - & - & 0 & a \\
	b   & - & - & - & - & - & - & 0 & - & b \\
	c   & - & - & - & - & - & 0 & - & - & c \\
	d   & - & - & - & - & 0 & - & b & - & d \\
	e   & - & - & - & 0 & - & - & a & b & e \\
	f   & - & - & 0 & - & b & a & d & c & f \\
	g   & - & 0 & - & - & - & b & c & - & g \\
	1   & 0 & a & b & c & d & e & f & g & 1
	\end{array}
\]
\[
\begin{array}{c|c|c|c|c|c|c|c|c|c}
\rightarrow &   0   &    a    &    b    &    c    &     d     &      e      &        f        &      g      &        1 \\
\hline
      0     & \{1\} &  \{1\}  &  \{1\}  &  \{1\}  &   \{1\}   &    \{1\}    &      \{1\}      &    \{1\}    &      \{1\} \\
\hline
			a     & \{g\} & \{g,1\} &  \{g\}  &  \{g\}  &   \{g\}   &   \{g,1\}   &     \{g,1\}     &    \{g\}    &     \{g,1\} \\
\hline
			b     & \{f\}	&  \{f\}	& \{f,1\}	&  \{f\}	&  \{f,1\}	&   \{f,1\}	  &     \{f,1\}	    &   \{f,1\}	  &     \{f,1\}	\\
\hline
			c     & \{e\} &  \{e\}  &  \{e\}  & \{e,1\} &   \{e\}   &    \{e\}    &     \{e,1\}     &   \{e,1\}   &     \{e,1\} \\
\hline
			d     & \{d\} &  \{d\}  & \{d,f\} &  \{d\}  & \{d,f,1\} &   \{d,f\}   &    \{d,f,1\}    &   \{d,f\}   &    \{d,f,1\} \\
\hline
			e     & \{c\} & \{c,f\} & \{c,g\} &  \{c\}  &  \{c,g\}  & \{c,f,g,1\} &    \{c,f,g\}    &   \{c,g\}   &   \{c,f,g,1\} \\
\hline
			f     & \{b\} & \{b,e\} & \{b,d\} & \{b,g\} & \{b,d,f\} &  \{b,d,e\}  & \{b,d,e,f,g,1\} &  \{b,d,g\}  & \{b,d,e,f,g,1\} \\
\hline
			g     & \{a\} &  \{a\}  & \{a,e\} & \{a,f\} &  \{a,e\}  &   \{a,e\}   &    \{a,e,f\}    & \{a,e,f,1\} &   \{a,e,f,1\} \\
\hline
			1     & \{0\} & \{0,a\} & \{0,b\} & \{0,c\} & \{0,b,d\} & \{0,a,b,e\} & \{0,a,b,c,d,f\} & \{0,b,c,g\} & E
\end{array}
\]
The operation $\rightarrow$ is unsharp since for $x,y\in E$, $x\rightarrow y$ need not be a singleton. One can see that the operator $\rightarrow$ is everywhere defined contrary to the fact that $\odot$ is only a partial operation. If $\mathbf E$ denotes the effect algebra from Example~\ref{ex2} then
\begin{align*}
      A\odot B & =A\cap B\text{ if and only if }A'\subseteq B, \\
A\rightarrow B & =\{A'\cup C\mid C\in E,C\subseteq A\cap B\}=\{D\in E\mid A'\subseteq D\subseteq A'\cup B\}
\end{align*}
for all $A,B\in E$.
\end{example}

Also, conversely, we can show that every commutative unsharp residuated poset can be organized into an effect algebra.

\begin{theorem}
Let $\mathbf C=(C,\leq,\odot,\rightarrow,{}',0,1)$ be a commutative unsharp residuated poset and put
\[
x+y:=(x'\odot y')'\text{ if and only if }x\leq y'
\]
{\rm(}$x,y\in C${\rm)}. Then $\mathbb E(\mathbf C):=(C,+,{}',0,1)$ is an effect algebra whose induced order coincides with the order in $\mathbf C$.
\end{theorem}

\begin{proof}
Let $a,b\in C$. Obviously, (E1), (E2) and (E4) hold.
\begin{enumerate}
\item[(E3)] Since
\[
LU(0,a')=LU(a')=L(a')\leq U(a')=U(a\rightarrow0)
\]
we have
\[
L(U(a')\odot a)=L(U(0,a')\odot a)\leq UL(a,0)=UL(0)=U(0)=C,
\]
i.e.\ $L(U(a')\odot a)=\{0\}$ and hence $a'\odot a\in L(U(a')\odot a)=\{0\}$, i.e.\ $a'\odot a=0$. If, conversely, $a\odot b=0$ then $a'\leq b$ and hence
\[
L(U(b,a')\odot a)=L(U(b)\odot a)=\{0\}\leq C=U(0)=UL(0)=UL(a,0)
\]
whence
\[
L(b)=LU(b)=LU(b,a')\leq U(a\rightarrow0)=U(a')
\]
showing $b\leq a'$ and hence $b=a'$. This shows that $a\odot b=0$ is equivalent to $b=a'$. Now the following are equivalent:
\begin{align*}
       a+b & =1, \\
a'\odot b' & =0, \\
        b' & =a, \\
         b & =a'.
\end{align*}
\end{enumerate}
Finally the following are equivalent:
\begin{align*}
& a\leq b\text{ in }\mathbb E(\mathbf C), \\
& a+b'\text{ is defined}, \\
& a'\odot b\text{ is defined}, \\
& a\leq b\text{ in }\mathbf C.
\end{align*}
Hence the induced order in $\mathbb E(\mathbf C)$ coincides with the order in $\mathbf C$.
\end{proof}

Every monotonous effect algebra can be reconstructed from its assigned commutative unsharp residuated poset as the following theorem shows.

\begin{theorem}
Let $\mathbf E$ be a monotonous effect algebra. Then $\mathbb E(\mathbb C(\mathbf E))=\mathbf E$.
\end{theorem}

\begin{proof}
Let
\begin{align*}
                      \mathbf E & =(E,+,{}',0,1)\text{ with induced order }\leq, \\
           \mathbb C(\mathbf E) & =(E,\leq,\odot,\rightarrow,{}',0,1), \\
\mathbb E(\mathbb C(\mathbf E)) & =(E,\oplus,{}',0,1)
\end{align*}
and $a,b\in E$. Then the following are equivalent:
\begin{align*}
& a\oplus b\text{ is defined}, \\
& a\leq b'\text{ in }\mathbb C(\mathbf E), \\
& a\leq b'\text{ in }\mathbf E, \\
& a+b\text{ is defined}
\end{align*}
and in this case
\[
a\oplus b=(a'\odot b')'=(a''+b'')''=a+b.
\]
\end{proof}

The concept of a pseudoeffect algebra was introduced by A.~Dvure\v censkij and T.~Vetterlein (\cite{DV}). The motivation for this concept and its connection to the logic of quantum mechanics is included in that paper. Our next goal is to show that also these algebras can be converted into some kind of unsharp residuated posets which, however, need not be commutative as in the case of effect algebras. We start with definition taken from \cite{DV}.

\begin{definition}
A {\em pseudoeffect algebra} is a partial algebra $\mathbf P=(P,+,\,\bar{}\,,\,\tilde{}\,,0,1)$ of type $(2,1,1,0,0)$ where $(P,\,\bar{}\,,\,\tilde{}\,,0,1)$ is an algebra and $+$ is a partial operation satisfying the following conditions for all $x,y,z\in P$:
\begin{enumerate}[{\rm(P1)}]
\item If $x+y$ is defined then there exist $u,w\in P$ with $u+x=y+w=x+y$,
\item $(x+y)+z$ is defined if and only if $x+(y+z)$ is defined, and in this case $(x+y)+z=x+(y+z)$,
\item $\bar x$ is the unique $u\in P$ with $u+x=1$, and $\tilde x$ is the unique $w\in P$ with $x+w=1$,
\item if $1+x$ or $x+1$ is defined then $x=0$.
\end{enumerate}
The {\em pseudoeffect algebra} $\mathbf P$ is called
\begin{itemize}
\item {\em good} if $\widetilde{\bar x+\bar y}=\overline{\tilde x+\tilde y}$ for all $x,y\in P$ with $\tilde x\leq y$,
\item {\em monotonous} if for all $x\in P$ and all non-empty subsets $A,B$ of $P$ the following hold:
\begin{itemize}
\item $A\cup B\leq\bar x$ and $L(A)\leq U(B)$ imply $L(A+x)\leq U(B+x)$,
\item $A\cup B\leq\tilde x$ and $L(A)\leq U(B)$ imply $L(x+A)\leq U(x+B)$.
\end{itemize}
\end{itemize}
\end{definition}

On $P$ a binary relation $\leq$ can be defined by
\[
x\leq y\text{ if there exists some }z\in E\text{ with }x+z=y
\]
($x,y\in P$). Then $(P,\leq,0,1)$ is a bounded poset and $\leq$ is called the {\em induced order} of $\mathbf P$.

For our investigation we need the following results taken from \cite{DV}.

\begin{lemma}
If $\mathbf P=(P,+,\,\bar{}\,,\,\tilde{}\,,0,1)$ is a pseudoeffect algebra, $\leq$ its induced order and $a,b,c\in P$ then
\begin{enumerate}[{\rm(i)}]
\item $\bar{\tilde a}=\tilde{\bar a}=a$,
\item $a\leq b$ implies $\bar b\leq\bar a$ and $\tilde b\leq\tilde a$,
\item $a+b$ is defined if and only if $a\leq\bar b$,
\item if $a\leq b$ and $b+c$ is defined then so is $a+c$ and $a+c\leq b+c$,
\item if $a\leq b$ and $c+b$ is defined then so is $c+a$ and $c+a\leq c+b$,
\item if $a\leq b$ then $a+\widetilde{\bar b+a}=\overline{a+\tilde b}+a=b$,
\item $a+0=0+a=a$,
\item $\bar0=\tilde0=1$ and $\bar1=\tilde1=0$,
\item the following are equivalent: $a\leq b$, there exists some $d\in P$ with $a+d=b$, there exists some $e\in P$ with $e+a=b$.
\end{enumerate}
\end{lemma}

Due to the lack of commutativity, we must modify the concept of an unsharp residuated poset as follows:

\begin{definition}
An {\em unsharp residuated poset} is an ordered ninetuple $\mathbf R=(R,\leq,\odot,\rightarrow,\leadsto,\,\bar{}\,,\,\tilde{}\,,0,1)$ where $\rightarrow,\leadsto:R^2\rightarrow2^R$ and $\,\bar{}\,$ and $\,\tilde{}\,$ are unary operations on $R$ and the following hold for all $x,y,z\in R$:
\begin{enumerate}[{\rm(R1)}]
\item $(R,\leq,0,1)$ is a bounded poset,
\item $\tilde{\bar x}\approx\bar{\tilde x}\approx x$, and $x\leq y$ implies $\bar y\leq\bar x$ and $\tilde y\leq\tilde x$,
\item $(R,\odot,1)$ is a partial monoid where $x\odot y$ is defined if and only if $\tilde x\leq y$. Moreover, $\bar z\leq x\leq y$ implies $x\odot z\leq y\odot z$, and $\tilde z\leq x\leq y$ implies $z\odot x\leq z\odot y$,
\item $L(U(x,\bar y)\odot y)\leq UL(y,z)$ if and only if $LU(x,\bar y)\leq U(y\rightarrow z)$,
\item $L(y\odot U(x,\tilde y))\leq UL(y,z)$ if and only if $LU(x,\tilde y)\leq U(y\leadsto z)$,
\item $x\rightarrow0\approx\{\bar x\}$ and $x\leadsto0\approx\{\tilde x\}$,
\item $\widetilde{\bar x\odot\bar y}\approx\overline{\tilde x\odot\tilde y}$.
\end{enumerate}
The unsharp residuated poset $\mathbf R$ is called {\em divisible} if $x\leq y$ implies $(y\rightarrow x)\odot y=y\odot(y\leadsto x)=L(x)$ {\rm(}$x,y\in R${\rm)}.
\end{definition}

Conditions (R4) and (R5) are together called {\em unsharp adjointness}. Observe that because of $\bar y\leq U(x,\bar y)$ and $\tilde y\leq U(x,\tilde y)$ the expressions $U(x,\bar y)\odot y$ and $y\odot U(x,\tilde y)$ are well-defined.

Although now the situation is more complicated as for effect algebras, we are still able to prove that every good monotonous pseudoeffect algebra can be converted into an unsharp residuated poset.

\begin{theorem}
Let $\mathbf P=(P,+,\,\bar{}\,,\,\tilde{}\,,0,1)$ be a good monotonous pseudoeffect algebra with induced order $\leq$ and put
\begin{align*}
      x\odot y & :=\widetilde{\bar x+\bar y}\text{ if and only if }\tilde x\leq y, \\
x\rightarrow y & :=\bar x+L(x,y), \\
   x\leadsto y & :=L(x,y)+\tilde x
\end{align*}
{\rm(}$x,y\in P${\rm)}. Then $\mathbb R(\mathbf P):=(P,\leq,\odot,\rightarrow,\leadsto,\,\bar{}\,,\,\tilde{}\,,0,1)$ is a divisible unsharp residuated poset.
\end{theorem}

\begin{proof}
Let $a,b,c\in P$. Obviously (R1), (R2), the first part of (R3) and (R7) hold.
\begin{enumerate}
\item[(R3)] If $\bar c\leq a\leq b$ then
\[
a\odot c=\widetilde{\bar a+\bar c}\leq\widetilde{\bar b+\bar c}=b\odot c,
\]
and if $\tilde c\leq a\leq b$ then
\[
c\odot a=\overline{\tilde c+\tilde a}\leq\overline{\tilde c+\tilde b}=c\odot b.
\]
\item[(R4)] If $L(U(a,\bar b)\odot b)\leq UL(b,c)$ then because of $\bar b\leq U(a,\bar b)$, $U(a,\bar b)\odot b\leq b$ and $L(b,c)\leq b$ we have
\[
LU(a,\bar b)=L(\bar b+\widetilde{\overline{U(a,\bar b)}+\bar b})=L(\bar b+(U(a,\bar b)\odot b))\leq U(\bar b+L(b,c))=U(b\rightarrow c).
\]
If, conversely, $LU(a,\bar b)\leq U(b\rightarrow c)$ then because of $\bar b\leq U(a,\bar b)$, $\bar b\leq b\rightarrow c$ and $L(b,c)\leq b$ we have
\[
L(U(a,\bar b)\odot b)\leq U((b\rightarrow c)\odot b)=U((\bar b+L(b,c))\odot b)=U(\widetilde{\overline{\bar b+L(b,c)}+\bar b})=UL(b,c).
\]
\item[(R5)] If $L(b\odot U(a,\tilde b))\leq UL(b,c)$ then because of $\tilde b\leq U(a,\tilde b)$, $b\odot U(a,\tilde b)\leq b$ and $L(b,c)\leq b$ we have
\[
LU(a,\tilde b)=L(\overline{\tilde b+\widetilde{U(a,\tilde b)}}+\tilde b)=L((b\odot U(a,\tilde b))+\tilde b)\leq U(L(b,c)+\tilde b)=U(b\leadsto c).
\]
If, conversely, $LU(a,\tilde b)\leq U(b\leadsto c)$ then because of $\tilde b\leq U(a,\tilde b)$, $\tilde b\leq b\leadsto c$ and $L(b,c)\leq b$ we have
\[
L(b\odot U(a,\tilde b))\leq U(b\odot(b\leadsto c))=U(b\odot(L(b,c)+\tilde b))=U(\overline{\tilde b+\widetilde{L(b,c)+\tilde b}})=UL(b,c).
\]
\item[(R6)] We have
\begin{align*}
x\rightarrow0 & \approx\bar x+L(x,0)\approx\{\bar x\}, \\
   x\leadsto0 & \approx L(x,0)+\tilde x\approx\{\tilde x\}.
\end{align*}
\end{enumerate}
If, finally, $a\leq b$ then $L(a)\leq b$ and hence
\begin{align*}
(b\rightarrow a)\odot b & =(\bar b+L(b,a))\odot b=\widetilde{\overline{\bar b+L(a)}+\bar b}=L(a), \\
    b\odot(b\leadsto a) & =b\odot(L(b,a)+\tilde b)=\overline{\tilde b+\widetilde{L(a)+\tilde b}}=L(a)
\end{align*}
proving divisibility.
\end{proof}

Surprisingly, also every unsharp residuated poset can be organized into a good pseudoeffect algebra.

\begin{theorem}
Let $\mathbf R=(R,\leq,\odot,\rightarrow,\leadsto,\,\bar{}\,,\,\tilde{}\,,0,1)$ be an unsharp residuated poset and put
\[
x+y:=\widetilde{\bar x\odot\bar y}\text{ if and only if }x\leq\bar y
\]
{\rm(}$x,y\in R${\rm)}. Then $\mathbb P(\mathbf R):=(R,+,\,\bar{}\,,\,\tilde{}\,,0,1)$ is a good pseudoeffect algebra whose induced order coincides with the order in $\mathbf R$.
\end{theorem}

\begin{proof}
Let $a,b,c\in R$. Obviously, (P2) and (P4) hold.
\begin{enumerate}
\item[(P3)] Since
\[
LU(0,\tilde a)=LU(\tilde a)=L(\tilde a)\leq U(\tilde a)=U(a\leadsto0)
\]
we have
\[
L(a\odot U(\tilde a))=L(a\odot U(0,\tilde a))\leq UL(a,0)=UL(0)=U(0)=R,
\]
i.e.\ $L(a\odot U(\tilde a))=\{0\}$ and hence $a\odot\tilde a\in L(a\odot U(\tilde a))=\{0\}$, i.e.\ $a\odot\tilde a=0$. If, conversely, $a\odot b=0$ then $\tilde a\leq b$ and hence
\[
L(a\odot U(b,\tilde a))=L(a\odot U(b))=\{0\}\leq R=U(0)=UL(0)=UL(a,0)
\]
whence
\[
L(b)=LU(b)=LU(b,\tilde a)\leq U(a\leadsto0)=U(\tilde a)
\]
showing $b\leq\tilde a$ and hence $b=\tilde a$. This shows that $a\odot b=0$ is equivalent to $b=\tilde a$. Now the following are equivalent:
\begin{align*}
& a+b=1, \\
& \bar a\odot\bar b=0, \\
& \bar b=a, \\
& a=\bar b.
\end{align*}
Since
\[
LU(0,\bar b)=LU(\bar b)=L(\bar b)\leq U(\bar b)=U(b\rightarrow0)
\]
we have
\[
L(U(\bar b)\odot b)=L(U(0,\bar b)\odot b)\leq UL(b,0)=UL(0)=U(0)=R,
\]
i.e.\ $L(U(\bar b)\odot b)=\{0\}$ and hence $\bar b\odot b\in L(U(\bar b)\odot b)=\{0\}$, i.e.\ $\bar b\odot b=0$. If, conversely, $a\odot b=0$ then $\bar b\leq a$ and hence
\[
L(U(a,\bar b)\odot b)=L(U(a)\odot b)=\{0\}\leq R=U(0)=UL(0)=UL(b,0)
\]
whence
\[
L(a)=LU(a)=LU(a,\bar b)\leq U(b\rightarrow0)=U(\bar b)
\]
showing $a\leq\bar b$ and hence $a=\bar b$. This shows that $a\odot b=0$ is equivalent to $a=\bar b$. Now the following are equivalent:
\begin{align*}
                  a+b & =1, \\
\tilde a\odot\tilde b & =0, \\
             \tilde a & =b, \\
                    b & =\tilde a.
\end{align*}
\item[(P1)] Since
\[
L(a\odot U(1,\tilde a))=L(a\odot U(1))=L(a)\leq U(a)=UL(a)=UL(a,a)
\]
we have
\[
R=L(1)=LU(1)=LU(1,\tilde a)\leq U(a\leadsto a),
\]
i.e.\ $U(a\leadsto a)=\{1\}$ and hence $a\leadsto a=1$. If $a\leq\bar b$ then because of $LU(\bar b,a)\leq U(\bar a\leadsto\bar a)$ we have
\[
L(\bar a\odot U(\bar b))=L(\bar a\odot U(\bar b,a))\leq UL(\bar a,\bar a)=UL(\bar a)=U(\bar a)
\]
whence $\bar a\odot\bar b\leq\bar a$ and therefore $a=\tilde{\bar a}\leq\widetilde{\bar a\odot\bar b}=a+b$ showing that $a+\widetilde{a+b}$ is defined. Now in case $a\leq\bar b$ the following are equivalent:
\begin{align*}
                    c & =\overline{a+\widetilde{a+b}}, \\
c+(a+\widetilde{a+b}) & =1, \\
(c+a)+\widetilde{a+b} & =1, \\
                  c+a & =a+b.
\end{align*}
Since
\[
L(U(1,\bar a)\odot a)=L(U(1)\odot a)=L(a)\leq U(a)=UL(a)=UL(a,a)
\]
we have
\[
R=L(1)=LU(1)=LU(1,\bar a)\leq U(a\rightarrow a),
\]
i.e.\ $U(a\rightarrow a)=\{1\}$ and hence $a\rightarrow a=1$. If $b\leq\tilde a$ then because of $LU(\tilde a,b)\leq U(\tilde b\rightarrow\tilde b)$ we have
\[
L(U(\tilde a)\odot\tilde b)=L(U(\tilde a,b)\odot\tilde b)\leq UL(\tilde b,\tilde b)=UL(\tilde b)=U(\tilde b)
\]
whence $\tilde a\odot\tilde b\leq\tilde b$ and therefore $b=\bar{\tilde b}\leq\overline{\tilde a\odot\tilde b}=a+b$ showing that $\overline{a+b}+b$ is defined. Now in case $b\leq\tilde a$ the following are equivalent:
\begin{align*}
                   c & =\widetilde{\overline{a+b}+b}, \\
(\overline{a+b}+b)+c & =1, \\
\overline{a+b}+(b+c) & =1, \\
                 b+c & =a+b.
\end{align*}
\end{enumerate}
Finally, the following are equivalent:
\begin{align*}
& a\leq b\text{ holds in }\mathbb P(\mathbf R), \\
& a+\tilde b\text{ is defined}, \\
& \bar a\odot b\text{ is defined}, \\
& a\leq b\text{ holds in }\mathbf R.
\end{align*}
Hence the induced order in $\mathbb P(\mathbf R)$ coincides with the order in $\mathbf R$.
\end{proof}

Similarly as in the case of effect algebras we can show that every good monotonous pseudoeffect algebra can be reconstructed from its assigned unsharp residuated poset.

\begin{theorem}
Let $\mathbf P$ be a good monotonous pseudoeffect algebra. Then $\mathbb P(\mathbb R(\mathbf P))=\mathbf P$.
\end{theorem}

\begin{proof}
Let
\begin{align*}
                      \mathbf P & =(P,+,\,\bar{}\,,\,\tilde{}\,,0,1)\text{ with induced order }\leq, \\
           \mathbb R(\mathbf P) & =(P,\leq,\odot,\rightarrow,\leadsto,\,\bar{}\,,\,\tilde{}\,,0,1), \\
\mathbb P(\mathbb R(\mathbf P)) & =(P,\oplus,\,\bar{}\,,\,\tilde{}\,,0,1)
\end{align*}
and $a,b\in P$. Then the following are equivalent:
\begin{align*}
& a\oplus b\text{ is defined}, \\
& a\leq\bar b\text{ in }\mathbb R(\mathbf P), \\
& a\leq\bar b\text{ in }\mathbf P, \\
& a+b\text{ is defined}
\end{align*}
and in this case
\[
a\oplus b=\widetilde{\bar a\odot\bar b}=\overline{\tilde a\odot\tilde b}=\overline{\widetilde{\bar{\tilde a}+\bar{\tilde b}}}=a+b.
\]
\end{proof}

Authors' addresses:

Ivan Chajda \\
Palack\'y University Olomouc \\
Faculty of Science \\
Department of Algebra and Geometry \\
17.\ listopadu 12 \\
771 46 Olomouc \\
Czech Republic \\
ivan.chajda@upol.cz

Helmut L\"anger \\
TU Wien \\
Faculty of Mathematics and Geoinformation \\
Institute of Discrete Mathematics and Geometry \\
Wiedner Hauptstra\ss e 8-10 \\
1040 Vienna \\
Austria, and \\
Palack\'y University Olomouc \\
Faculty of Science \\
Department of Algebra and Geometry \\
17.\ listopadu 12 \\
771 46 Olomouc \\
Czech Republic \\
helmut.laenger@tuwien.ac.at

\begin{thebibliography}9
\bibitem{CH}
I.~Chajda and R.~Hala\v s, Effect algebras are conditionally residuated structures. Soft Computing {\bf15} (2011), 1383--1387.
\bibitem{CL17a}
I.~Chajda and H.~L\"anger, Residuation in orthomodular lattices. Topol.\ Algebra Appl.\ {\bf5} (2017), 1--5.
\bibitem{CL17b}
I.~Chajda and H.~L\"anger, Orthomodular lattices can be converted into left residuated l-groupoids. Miskolc Math.\ Notes {\bf18} (2017), 685--689.
\bibitem{CL1}
I.~Chajda and H.~L\"anger, Relatively residuated lattices and posets, Math.\ Slovaca (submitted). http://arxiv.org/abs/1901.06664.
\bibitem{CL2}
I.~Chajda and H.~L\"anger, Residuation in lattice effect algebras. Fuzzy Sets Systems (submitted). http://arxiv.org/abs/1905.05496.
\bibitem{DP}
A.~Dvure\v censkij and S.~Pulmannov\'a, New Trends in Quantum Structures. Kluwer, Dordrecht 2000. ISBN 0-7923-6471-6.
\bibitem{DV}
A.~Dvure\v censkij and T.~Vetterlein, Pseudoeffect algebras. I. Basic properties. Internat.\ J.\ Theoret.\ Phys.\ {\bf40} (2001), 685--701.
\bibitem{FB}
D.~J.~Foulis and M.~K.~Bennett, Effect algebras and unsharp quantum logics. Found.\ Phys.\ {\bf24} (1994), 1331--1352.
\bibitem{GJKO}
N.~Galatos, P.~Jipsen, T.~Kowalski and H.~Ono, Residuated Lattices: An Algebraic Glimpse at Substructural Logics, Elsevier, Amsterdam 2007. ISBN 978-0-444-52141-5.
\end{thebibliography}
\end{document}